\documentclass[11pt]{amsart}
\usepackage[a4paper,margin=28mm]{geometry}
\usepackage[T1]{fontenc}
\usepackage{lmodern}
\usepackage{amsmath,amssymb,amsthm,mathtools}
\usepackage{booktabs,array,longtable}
\usepackage{microtype}
\usepackage{xcolor}
\usepackage{tikz}
\usepackage{hyperref}
\usepackage[nameinlink,capitalize,noabbrev]{cleveref}
\usepackage{enumitem}
\hypersetup{colorlinks=true,linkcolor=blue!55!black,citecolor=blue!55!black,urlcolor=blue!55!black,pdftitle={Exact Zarankiewicz Values on Two Finite Frontier Slices},pdfauthor={Koyar Afrasyab}}

\newtheorem{theorem}{Theorem}[section]
\newtheorem{lemma}[theorem]{Lemma}
\newtheorem{proposition}[theorem]{Proposition}

\newtheorem{definition}[theorem]{Definition}
\newtheorem{remark}[theorem]{Remark}
\newcommand{\Z}{\mathrm Z}

\DeclareMathOperator{\rank}{rank}

\title[Exact Zarankiewicz values on two finite frontier slices]{Exact Zarankiewicz Values on Two Finite Frontier Slices}
\author{Koyar Afrasyab}
\thanks{Independent researcher. Email: \texttt{koyar@kinvectum.com}.}
\date{8 August 2026}

\begin{document}
\begin{abstract}
The Zarankiewicz number $\Z(m,n,s,t)$ is the maximum number of edges in a bipartite graph with parts of orders $m$ and $n$ containing no copy of $K_{s,t}$. We give one combined, certificate-based computer-assisted proof for two finite slices and a corrected neighboring frontier:
\[
\begin{gathered}
  \Z(12,n,3,3)=6n\quad(18\le n\le22),
  \qquad \Z(13,22,3,3)=137,\\
  \Z(13,18,3,3)=116,\qquad \Z(14,17,3,3)=118,\\
  \Z(14,18,3,3)=124,\qquad \Z(15,17,3,3)=126,\\
  \Z(15,18,3,3)=132,\qquad 132\le\Z(16,17,3,3)\le133.
\end{gathered}
\]
The load-bearing new upper bounds are the exact $12\times18$ and $13\times18$ certificate packages. Their orbit certificates exclude every hypothetical matrix at the next edge count. Deletion lemmas and explicit witnesses close four neighboring cells, while the $16\times17$ entry is deliberately reported as an interval because only its 132-edge lower witness and the published 133 upper bound are certified here. Separately, the $13\times22$ proof excludes 138 ones by reducing to 83 degree profiles, rationally separating 77 of them, and eliminating the remaining six by marked-row congruences, leave enumeration, modular Gram tests, and exact Farkas certificates. All accepted claims are replayed by standard-library Python and exact integer/rational arithmetic; floating-point optimization is used only to discover certificates.
\end{abstract}
\maketitle

\section{Introduction}
The Zarankiewicz problem asks for the maximum number of edges in a bipartite graph avoiding a specified complete bipartite subgraph. It was posed by Zarankiewicz in 1951 \cite{Zarankiewicz1951} and was placed in its modern extremal form by K\H{o}v\'ari, S\'os and Tur\'an \cite{KST1954}. Although the asymptotic theory is extensive, exact values for moderate finite parameters remain difficult. The difficulty is particularly pronounced for $K_{3,3}$-free graphs: useful counting inequalities leave a small but highly structured set of integer configurations, while direct exhaustive search suffers from severe symmetry and proof-certification costs.

Write $\Z(m,n,s,t)$ for the maximum number of edges in a subgraph of $K_{m,n}$ containing no $K_{s,t}$. Roman's linear-programming framework \cite{Roman1975} and later strengthened relaxations \cite{DaviesGillHorsley2026} provide effective finite upper bounds. SAT-based methods have also been developed for exact finite cases \cite{Tan2022}. The classical finite table of Collins, Riasanovsky, Wallace and Radziszowski \cite{CollinsRiasanovskyWallaceRadziszowski2016} supplies, among other inputs, $\Z(13,17,3,3)\le110$ and $\Z(16,17,3,3)\le133$. On the lower-bound side, Bhan, Nobili, Raghuraman and Langer \cite{BhanNobiliLanger2026} reported the 12-by-22 value and a 137-edge construction for the 13-by-22 parameter, with upper bound 140.
\[
  137\le \Z(13,22,3,3)\le140.
\]
The present paper adds the exact 12-by-18 certificate and its deletion consequences, and completes the 13-by-22 upper bound.

\begin{theorem}\label{thm:combined}
\begin{align*}
\Z(12,n,3,3)&=6n &&(18\le n\le22),\\
\Z(13,22,3,3)&=137,\\
\Z(13,18,3,3)&=116,\\
\Z(14,17,3,3)&=118, & \Z(14,18,3,3)&=124,\\
\Z(15,17,3,3)&=126, & \Z(15,18,3,3)&=132,\\
132\le\Z(16,17,3,3)&\le133.
\end{align*}
The values at $n=18,19,20,21$, the $13$-by-$22$ upper bound, the four
neighboring frontier equalities, and the 132-edge $16$-by-$17$ construction
are the claims assembled here. The $n=22$ equality and its classical design
are retained as prior material.
\end{theorem}

The proof is computer-assisted but certificate based. It is not a report of a mixed-integer solver returning ``infeasible.'' Every load-bearing numerical assertion is replayed by explicit finite data and exact arithmetic. The 12-by-$n$ strip, the 13-by-$22$ theorem, and the corrected frontier package retain separate certificate directories under \texttt{proof/z12\_18\_21/}, \texttt{proof/}, and \texttt{proof/frontier\_closure/}; the combined verification gate runs the first two and the frontier package's own gate.

\begin{theorem}\label{thm:13}
There is no $13\times22$ binary matrix with $138$ ones and no all-one $3\times3$ submatrix. Consequently,
\[
  \boxed{\Z(13,22,3,3)=137}.
\]
\end{theorem}

For the 13-by-22 component, the proof is computer-assisted but certificate based. It is not a report of a mixed-integer solver returning ``infeasible.'' Every load-bearing numerical assertion is replayed by short exact-arithmetic programs from explicit finite data. The architecture has three layers.

\begin{enumerate}[label=(\roman*)]
\item A global degree-profile argument reduces $138$ edges to $83$ profiles and eliminates $77$ by exact rational Farkas certificates.
\item A direct marked-row congruence argument excludes the exceptional profile $6^{18}7^3 9^1$.
\item The remaining five profiles are reduced to finite residual block-design problems on twelve points. Exhaustive leave-multigraph enumeration, necessary Gram conditions, and exact Farkas certificates exclude all residual cases.
\end{enumerate}

The source archive includes both certificate systems, all nested witnesses, independent profile enumerators, and a one-command verification gate. The result should nevertheless be regarded as a new computer-assisted proof pending independent reproduction and peer review.

\section{The exact $12\times n$ strip}
\label{sec:12-strip}

We first isolate the 12-row result. Identify each column with a support in
$[12]$ and let $\lambda_T$ be the number of columns containing a row triple
$T$. The $K_{3,3}$-free condition is $\lambda_T\le2$ for every
$T\in\binom{[12]}3$.

The load-bearing upper bound is the exact certificate proof
\[
  \Z(12,18,3,3)\le108.
\]
Its self-contained verifier enumerates 303 degree profiles for the neighboring
$12\times17$ bound and 51 profiles for the $11\times18$ bound. Those bounds
force any hypothetical 109-one $12\times18$ matrix into four normalized
containment cases. The orbit trees in \texttt{proof/z12\_18\_21/} contain exact
rational dual leaves and exhaustive integer branches for all four cases; the
standard-library replay accepts every certificate.

The 18-column witness has 108 ones and row-triple histogram
$(6,68,146)$ for multiplicities $(0,1,2)$. Four further witnesses are the
first 19, 20, 21, and 22 blocks of the explicit classical 3-$(12,6,2)$
Hadamard design listed in Appendix~\ref{app:12-design}; their histograms are
$(3,54,163)$, $(1,38,181)$, $(0,20,200)$, and $(0,0,220)$, respectively.
Thus the lower bounds are $108,114,120,126,132$.

\begin{lemma}[Minimum-column deletion]
\label{lem:min-column}
For every $m,n,s,t$ with $n\ge2$,
\[
  \Z(m,n,s,t)\le
  \left\lfloor\frac{n}{n-1}\,\Z(m,n-1,s,t)\right\rfloor.
\]
\end{lemma}
\begin{proof}
If an $e$-edge $m\times n$ matrix existed, some column would have degree at
most $\lfloor e/n\rfloor$. Deleting it leaves at least
$e-\lfloor e/n\rfloor$ edges in an admissible $m\times(n-1)$ matrix. Taking
$e$ one larger than the displayed floor makes this residual count exceed
$\Z(m,n-1,s,t)$, a contradiction.
\end{proof}

Applying the lemma successively from $\Z(12,18,3,3)=108$ gives
\[
\Z(12,19,3,3)\le114,\quad
\Z(12,20,3,3)\le120,\quad
\Z(12,21,3,3)\le126,\quad
\Z(12,22,3,3)\le132.
\]
The nested witnesses attain all four bounds, proving the first line of
Theorem~\ref{thm:combined}. The $n=22$ equality is included for a unified
statement but is not claimed as a new value.

\section{Corrected finite-frontier closure}
\label{sec:frontier}

The second reproducibility package, under
\texttt{proof/frontier\_closure/}, proves the new exact value
\[
  \Z(13,18,3,3)=116.
\]
Its verifier reduces a hypothetical 117-edge matrix to 19 column-degree
profiles. Thirteen profiles are separated by exact integer-scaled Farkas
certificates; the remaining six are eliminated by marked-row deficit
averaging, exact point-type enumeration, and 12 symmetry-orbit pair-packing
certificates. The packaged 116-edge witness is checked exhaustively over all
row triples. The package's optional dependency replay rechecks the earlier
12-by-17 and 12-by-18 certificates.

Four neighboring equalities follow by the one-row density lemma and explicit
witnesses. Collins et al.~\cite{CollinsRiasanovskyWallaceRadziszowski2016}
give the published bound $\Z(13,17,3,3)\le110$, so a 119-edge $14\times17$
matrix would leave at least
\[
119-\left\lfloor119/14\right\rfloor=111>110
\]
edges after deleting a row; hence $\Z(14,17)\le118$. Similarly,
\[
125-\left\lfloor125/14\right\rfloor=117>116,
\qquad
127-\left\lfloor127/15\right\rfloor=119>118,
\]
give the upper bounds $\Z(14,18)\le124$ and $\Z(15,17)\le126$.
Finally,
\[
133-\left\lfloor133/15\right\rfloor=125>124
\]
gives $\Z(15,18)\le132$. The corresponding 118-, 124-, 126-, and 132-edge
witnesses are stored and checked in the same package.

For $16\times17$, the supplied 132-edge witness gives the lower bound
$\Z(16,17)\ge132$, while the published table gives $\Z(16,17)\le133$.
We therefore record only
\[
  132\le\Z(16,17,3,3)\le133.
\]
No SAT timeout, unsuccessful construction search, or unverified DRAT claim
is used in this interval.

\begin{remark}[Boundary of the retained frontier claims]
If $c_{ij}$ is the number of columns containing both rows $i$ and $j$,
deleting those rows removes $d_i+d_j-c_{ij}$ edges, so the remaining edge
count is $e-d_i-d_j+c_{ij}$. Earlier exploratory certificates based on this
reduction are not retained here: the explicit 126-edge $15\times17$ and
132-edge $16\times17$ witnesses rule out the former $125$- and $130$-edge
exact claims, and no independently replayed certificate for the associated
upper-bound claim is included. The paper therefore records only the certified
statements in Theorem~\ref{thm:combined}.
\end{remark}

\section{Block formulation and the lower bound}
Let $X=[13]$. Identify the $j$th column of a binary $13\times22$ matrix with its support $E_j\subseteq X$, and put $d_j=|E_j|$. For each triple $T\in\binom{X}{3}$, define
\[
  \lambda_T=|\{j:T\subseteq E_j\}|.
\]
The matrix is $K_{3,3}$-free if and only if
\begin{equation}\label{eq:triple-capacity}
  \lambda_T\le2 \qquad(T\in\binom{X}{3}).
\end{equation}
Repeated columns are allowed throughout.

The supplementary file \texttt{proof/data/z13\_22\_137\_blocks.json} contains $22$ blocks with total size $137$. Direct enumeration of all $\binom{13}{3}=286$ triples verifies \eqref{eq:triple-capacity}. Thus
\begin{equation}\label{eq:lower}
  \Z(13,22,3,3)\ge137.
\end{equation}
The same lower bound was first reported in \cite{BhanNobiliLanger2026}. A matrix plot of the independently checked witness is shown in \cref{fig:witness}; its block supports are listed in \cref{app:witness}.

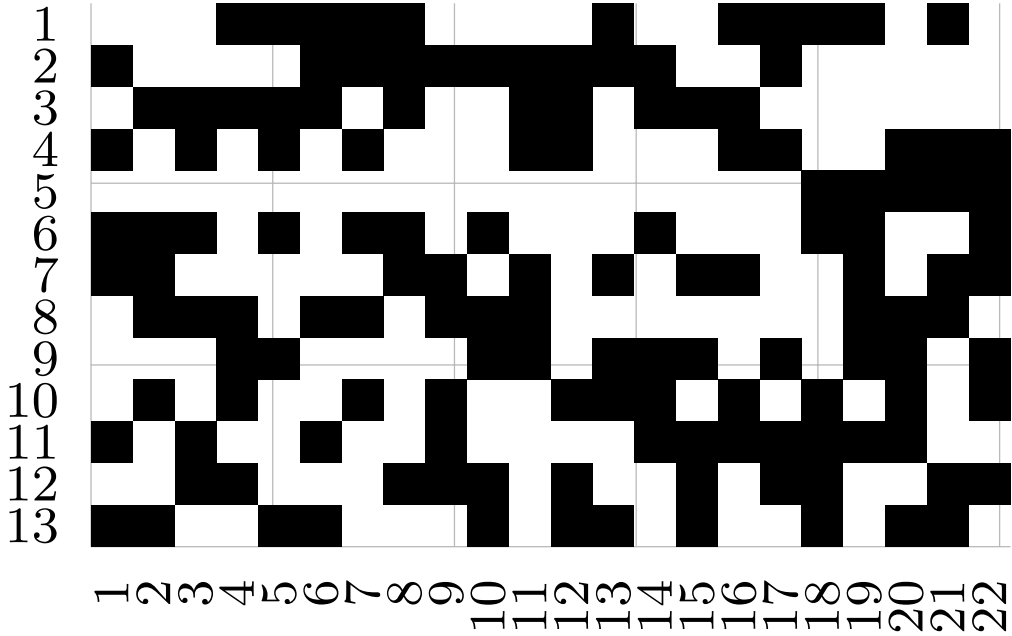
\begin{figure}[ht]
\centering
\resizebox{0.92\textwidth}{!}{\begin{tikzpicture}[x=0.23cm,y=0.23cm]
  \draw[gray!55,very thin] (0,0) grid (22,13);
  \fill (0,11) rectangle ++(1,1);
  \fill (0,9) rectangle ++(1,1);
  \fill (0,7) rectangle ++(1,1);
  \fill (0,6) rectangle ++(1,1);
  \fill (0,2) rectangle ++(1,1);
  \fill (0,0) rectangle ++(1,1);
  \fill (1,10) rectangle ++(1,1);
  \fill (1,7) rectangle ++(1,1);
  \fill (1,6) rectangle ++(1,1);
  \fill (1,5) rectangle ++(1,1);
  \fill (1,3) rectangle ++(1,1);
  \fill (1,0) rectangle ++(1,1);
  \fill (2,10) rectangle ++(1,1);
  \fill (2,9) rectangle ++(1,1);
  \fill (2,7) rectangle ++(1,1);
  \fill (2,5) rectangle ++(1,1);
  \fill (2,2) rectangle ++(1,1);
  \fill (2,1) rectangle ++(1,1);
  \fill (3,12) rectangle ++(1,1);
  \fill (3,10) rectangle ++(1,1);
  \fill (3,5) rectangle ++(1,1);
  \fill (3,4) rectangle ++(1,1);
  \fill (3,3) rectangle ++(1,1);
  \fill (3,1) rectangle ++(1,1);
  \fill (4,12) rectangle ++(1,1);
  \fill (4,10) rectangle ++(1,1);
  \fill (4,9) rectangle ++(1,1);
  \fill (4,7) rectangle ++(1,1);
  \fill (4,4) rectangle ++(1,1);
  \fill (4,0) rectangle ++(1,1);
  \fill (5,12) rectangle ++(1,1);
  \fill (5,11) rectangle ++(1,1);
  \fill (5,10) rectangle ++(1,1);
  \fill (5,5) rectangle ++(1,1);
  \fill (5,2) rectangle ++(1,1);
  \fill (5,0) rectangle ++(1,1);
  \fill (6,12) rectangle ++(1,1);
  \fill (6,11) rectangle ++(1,1);
  \fill (6,9) rectangle ++(1,1);
  \fill (6,7) rectangle ++(1,1);
  \fill (6,5) rectangle ++(1,1);
  \fill (6,3) rectangle ++(1,1);
  \fill (7,12) rectangle ++(1,1);
  \fill (7,11) rectangle ++(1,1);
  \fill (7,10) rectangle ++(1,1);
  \fill (7,7) rectangle ++(1,1);
  \fill (7,6) rectangle ++(1,1);
  \fill (7,1) rectangle ++(1,1);
  \fill (8,11) rectangle ++(1,1);
  \fill (8,6) rectangle ++(1,1);
  \fill (8,5) rectangle ++(1,1);
  \fill (8,3) rectangle ++(1,1);
  \fill (8,2) rectangle ++(1,1);
  \fill (8,1) rectangle ++(1,1);
  \fill (9,11) rectangle ++(1,1);
  \fill (9,7) rectangle ++(1,1);
  \fill (9,5) rectangle ++(1,1);
  \fill (9,4) rectangle ++(1,1);
  \fill (9,1) rectangle ++(1,1);
  \fill (9,0) rectangle ++(1,1);
  \fill (10,11) rectangle ++(1,1);
  \fill (10,10) rectangle ++(1,1);
  \fill (10,9) rectangle ++(1,1);
  \fill (10,6) rectangle ++(1,1);
  \fill (10,5) rectangle ++(1,1);
  \fill (10,4) rectangle ++(1,1);
  \fill (11,11) rectangle ++(1,1);
  \fill (11,10) rectangle ++(1,1);
  \fill (11,9) rectangle ++(1,1);
  \fill (11,3) rectangle ++(1,1);
  \fill (11,1) rectangle ++(1,1);
  \fill (11,0) rectangle ++(1,1);
  \fill (12,12) rectangle ++(1,1);
  \fill (12,11) rectangle ++(1,1);
  \fill (12,6) rectangle ++(1,1);
  \fill (12,4) rectangle ++(1,1);
  \fill (12,3) rectangle ++(1,1);
  \fill (12,0) rectangle ++(1,1);
  \fill (13,11) rectangle ++(1,1);
  \fill (13,10) rectangle ++(1,1);
  \fill (13,7) rectangle ++(1,1);
  \fill (13,4) rectangle ++(1,1);
  \fill (13,3) rectangle ++(1,1);
  \fill (13,2) rectangle ++(1,1);
  \fill (14,10) rectangle ++(1,1);
  \fill (14,6) rectangle ++(1,1);
  \fill (14,4) rectangle ++(1,1);
  \fill (14,2) rectangle ++(1,1);
  \fill (14,1) rectangle ++(1,1);
  \fill (14,0) rectangle ++(1,1);
  \fill (15,12) rectangle ++(1,1);
  \fill (15,10) rectangle ++(1,1);
  \fill (15,9) rectangle ++(1,1);
  \fill (15,6) rectangle ++(1,1);
  \fill (15,3) rectangle ++(1,1);
  \fill (15,2) rectangle ++(1,1);
  \fill (16,12) rectangle ++(1,1);
  \fill (16,11) rectangle ++(1,1);
  \fill (16,9) rectangle ++(1,1);
  \fill (16,4) rectangle ++(1,1);
  \fill (16,2) rectangle ++(1,1);
  \fill (16,1) rectangle ++(1,1);
  \fill (17,12) rectangle ++(1,1);
  \fill (17,8) rectangle ++(1,1);
  \fill (17,7) rectangle ++(1,1);
  \fill (17,3) rectangle ++(1,1);
  \fill (17,2) rectangle ++(1,1);
  \fill (17,1) rectangle ++(1,1);
  \fill (17,0) rectangle ++(1,1);
  \fill (18,12) rectangle ++(1,1);
  \fill (18,8) rectangle ++(1,1);
  \fill (18,7) rectangle ++(1,1);
  \fill (18,6) rectangle ++(1,1);
  \fill (18,5) rectangle ++(1,1);
  \fill (18,4) rectangle ++(1,1);
  \fill (18,2) rectangle ++(1,1);
  \fill (19,9) rectangle ++(1,1);
  \fill (19,8) rectangle ++(1,1);
  \fill (19,5) rectangle ++(1,1);
  \fill (19,4) rectangle ++(1,1);
  \fill (19,3) rectangle ++(1,1);
  \fill (19,2) rectangle ++(1,1);
  \fill (19,0) rectangle ++(1,1);
  \fill (20,12) rectangle ++(1,1);
  \fill (20,9) rectangle ++(1,1);
  \fill (20,8) rectangle ++(1,1);
  \fill (20,6) rectangle ++(1,1);
  \fill (20,5) rectangle ++(1,1);
  \fill (20,1) rectangle ++(1,1);
  \fill (20,0) rectangle ++(1,1);
  \fill (21,9) rectangle ++(1,1);
  \fill (21,8) rectangle ++(1,1);
  \fill (21,7) rectangle ++(1,1);
  \fill (21,6) rectangle ++(1,1);
  \fill (21,4) rectangle ++(1,1);
  \fill (21,3) rectangle ++(1,1);
  \fill (21,1) rectangle ++(1,1);
  \node[font=\scriptsize,rotate=90,anchor=east] at (0.5,-0.2) {1};
  \node[font=\scriptsize,rotate=90,anchor=east] at (1.5,-0.2) {2};
  \node[font=\scriptsize,rotate=90,anchor=east] at (2.5,-0.2) {3};
  \node[font=\scriptsize,rotate=90,anchor=east] at (3.5,-0.2) {4};
  \node[font=\scriptsize,rotate=90,anchor=east] at (4.5,-0.2) {5};
  \node[font=\scriptsize,rotate=90,anchor=east] at (5.5,-0.2) {6};
  \node[font=\scriptsize,rotate=90,anchor=east] at (6.5,-0.2) {7};
  \node[font=\scriptsize,rotate=90,anchor=east] at (7.5,-0.2) {8};
  \node[font=\scriptsize,rotate=90,anchor=east] at (8.5,-0.2) {9};
  \node[font=\scriptsize,rotate=90,anchor=east] at (9.5,-0.2) {10};
  \node[font=\scriptsize,rotate=90,anchor=east] at (10.5,-0.2) {11};
  \node[font=\scriptsize,rotate=90,anchor=east] at (11.5,-0.2) {12};
  \node[font=\scriptsize,rotate=90,anchor=east] at (12.5,-0.2) {13};
  \node[font=\scriptsize,rotate=90,anchor=east] at (13.5,-0.2) {14};
  \node[font=\scriptsize,rotate=90,anchor=east] at (14.5,-0.2) {15};
  \node[font=\scriptsize,rotate=90,anchor=east] at (15.5,-0.2) {16};
  \node[font=\scriptsize,rotate=90,anchor=east] at (16.5,-0.2) {17};
  \node[font=\scriptsize,rotate=90,anchor=east] at (17.5,-0.2) {18};
  \node[font=\scriptsize,rotate=90,anchor=east] at (18.5,-0.2) {19};
  \node[font=\scriptsize,rotate=90,anchor=east] at (19.5,-0.2) {20};
  \node[font=\scriptsize,rotate=90,anchor=east] at (20.5,-0.2) {21};
  \node[font=\scriptsize,rotate=90,anchor=east] at (21.5,-0.2) {22};
  \node[font=\scriptsize,anchor=east] at (-0.18,12.5) {1};
  \node[font=\scriptsize,anchor=east] at (-0.18,11.5) {2};
  \node[font=\scriptsize,anchor=east] at (-0.18,10.5) {3};
  \node[font=\scriptsize,anchor=east] at (-0.18,9.5) {4};
  \node[font=\scriptsize,anchor=east] at (-0.18,8.5) {5};
  \node[font=\scriptsize,anchor=east] at (-0.18,7.5) {6};
  \node[font=\scriptsize,anchor=east] at (-0.18,6.5) {7};
  \node[font=\scriptsize,anchor=east] at (-0.18,5.5) {8};
  \node[font=\scriptsize,anchor=east] at (-0.18,4.5) {9};
  \node[font=\scriptsize,anchor=east] at (-0.18,3.5) {10};
  \node[font=\scriptsize,anchor=east] at (-0.18,2.5) {11};
  \node[font=\scriptsize,anchor=east] at (-0.18,1.5) {12};
  \node[font=\scriptsize,anchor=east] at (-0.18,0.5) {13};
\end{tikzpicture}}
\caption{A checked $13\times22$ matrix with $137$ ones. Black cells are ones. Every triple of rows is contained in at most two columns.}
\label{fig:witness}
\end{figure}

\section{Global reduction at 138 edges}
Assume for contradiction that blocks $E_1,\dots,E_{22}$ satisfy \eqref{eq:triple-capacity} and
\begin{equation}\label{eq:sumdegree}
  \sum_{j=1}^{22}d_j=138.
\end{equation}
Double counting incidences between blocks and row triples gives
\begin{equation}\label{eq:triplecount}
  \sum_{j=1}^{22}\binom{d_j}{3}
  =\sum_{T\in\binom X3}\lambda_T
  \le2\binom{13}{3}=572.
\end{equation}

Define the integer penalty
\begin{equation}\label{eq:penalty}
  p(d)=\binom d3-15d+70,\qquad 0\le d\le13.
\end{equation}
Its values are
\[
\begin{array}{c|rrrrrrrrrrrrrr}
d&0&1&2&3&4&5&6&7&8&9&10&11&12&13\\\hline
p(d)&70&55&40&26&14&5&0&0&6&19&40&70&110&161.
\end{array}
\]
Thus $p(d)\ge0$, with equality precisely for $d=6,7$. Combining \eqref{eq:sumdegree}--\eqref{eq:penalty},
\begin{equation}\label{eq:penaltybudget}
  \sum_{j=1}^{22}p(d_j)
  =\sum_j\binom{d_j}{3}-15\cdot138+70\cdot22
  \le42.
\end{equation}
Exact integer enumeration of all histograms $(n_0,\dots,n_{13})$ satisfying
\[
  \sum_dn_d=22,
  \qquad \sum_ddn_d=138,
  \qquad \sum_dp(d)n_d\le42
\]
leaves exactly $83$ degree profiles. This is independently regenerated by \texttt{enumerate138.py} and \texttt{verify\_138\_reduction.py}; no degree window is assumed.

\subsection{The block-polytope relaxation}
For a profile $(n_d)$, choose a degree $f$ with $n_f>0$ and, by row symmetry, fix one $f$-block as $F=\{1,\dots,f\}$. For every remaining allowed support $B\subseteq X$, let $x_B\ge0$ denote its multiplicity. Every actual matrix yields an integral feasible point of the following relaxation:
\begin{align}
 \sum_{B\supseteq T}x_B
 &\le2-\mathbf1[T\subseteq F]
 &&(T\in\binom X3), \label{eq:lp-triple}\\
 \sum_{B\ni r}\binom{|B|-1}{2}x_B
 &\le132-\mathbf1[r\in F]\binom{f-1}{2}
 &&(r\in X), \label{eq:lp-row}\\
 \sum_{B\supseteq P}(|B|-2)x_B
 &\le22-\mathbf1[P\subseteq F](f-2)
 &&(P\in\binom X2), \label{eq:lp-pair}\\
 \sum_{|B|=d}x_B
 &=n_d-\mathbf1[d=f]
 &&(d\text{ active}). \label{eq:lp-count}
\end{align}
The row and pair inequalities are nonnegative sums of triple-capacity inequalities; they are retained because they yield much shorter dual certificates.

\begin{lemma}[Exact Farkas criterion]\label{lem:farkas}
Write \eqref{eq:lp-triple}--\eqref{eq:lp-pair} as $Ax\le b$, and \eqref{eq:lp-count} as $Ex=e$, with $x\ge0$. If rational vectors $\alpha\ge0$ and unrestricted $\beta$ satisfy
\[
 A^T\alpha+E^T\beta\ge0,
 \qquad b^T\alpha+e^T\beta<0,
\]
then the profile is impossible.
\end{lemma}
\begin{proof}
For a feasible $x$,
\[
0\le x^T(A^T\alpha+E^T\beta)
=\alpha^TAx+\beta^TEx
\le\alpha^Tb+\beta^Te<0,
\]
a contradiction. This is the standard alternative theorem of Farkas \cite{Farkas1902}.
\end{proof}

The verifier regenerates every subset $B$ of each active degree and checks every dual coefficient using \texttt{fractions.Fraction}. Stored rational certificates exclude $77$ of the $83$ profiles. The six not separated by this global relaxation are
\begin{equation}\label{eq:sixprofiles}
\begin{split}
&6^{18}7^3 9^1,\qquad 6^{16}7^6,\qquad 5^1 6^{14}7^7,\qquad\
5^2 6^{12}7^8,\qquad 4^1 6^{13}7^8,\qquad 5^3 6^{10}7^9.
\end{split}
\end{equation}

\section{Marked-row deficits}
For every triple $T$, put
\[
  \delta_T=2-\lambda_T\ge0,
  \qquad
  D_r=\sum_{T\ni r}\delta_T.
\]
If
\[
 s=572-\sum_j\binom{d_j}{3}
\]
is the unused triple capacity, then
\begin{equation}\label{eq:deficitsum}
  \sum_{r\in X}D_r=3s.
\end{equation}
Since a fixed row belongs to $\binom{12}{2}=66$ triples, each of capacity two,
\begin{equation}\label{eq:deficitformula}
  D_r=132-\sum_{j:r\in E_j}\binom{d_j-1}{2}.
\end{equation}
This formula gives both congruence restrictions and a local design interpretation.

\subsection{The profile $6^{18}7^3 9^1$}
For this profile, $s=23$, so $\sum_rD_r=69$. Contributions from degree-six and degree-seven blocks in \eqref{eq:deficitformula} are $10$ and $15$, both zero modulo five; the degree-nine contribution is $28\equiv3\pmod5$. Thus
\[
 D_r\equiv
 \begin{cases}
 2\pmod5,&r\notin E_9,\\
 4\pmod5,&r\in E_9.
 \end{cases}
\]
The residue-minimal total is $4\cdot2+9\cdot4=44$. Since the actual total is $69$, only five increments of size five are available, and at least eight rows have residue-minimal deficit.

For a row outside the degree-nine block with $D_r=2$, if $a,b$ count degree-six and degree-seven blocks through $r$, then \eqref{eq:deficitformula} gives $2a+3b=26$, hence $(a,b)=(13,0)$ or $(10,2)$. Deleting $r$ leaves blocks of sizes five and six on twelve points. Pair-capacity degrees imply, respectively, total size-five incidence at most $60<65$, or at most $48<50$. Both are impossible. Therefore all four rows outside $E_9$ consume an increment.

At least eight rows inside $E_9$ consequently have $D_r=4$. For such a row, $2a+3b=20$. The case $(a,b)=(10,0)$ is ruled out by the residual degree-nine block: its eight points permit at most three size-five incidences and the other four at most five, totaling $44<50$. Hence each of these at least eight rows lies in exactly two of the three degree-seven columns. Yet any fixed pair of degree-seven columns can share at most two such rows inside $E_9$, since three shared rows together with $E_9$ would form a $K_{3,3}$. The three pairs account for at most six rows, a contradiction.

\begin{proposition}\label{prop:nine}
The profile $6^{18}7^3 9^1$ is impossible.
\end{proposition}

\section{The marked-row leave method}
It remains to exclude the five profiles in \eqref{eq:sixprofiles} containing only degrees four through seven. Fix a row $r$ and delete it from every block containing it. The resulting residual blocks lie on the twelve-point set $X\setminus\{r\}$.

\begin{definition}
The \emph{leave multigraph} $L_r$ has vertex set $X\setminus\{r\}$ and gives the pair $\{x,y\}$ multiplicity $\delta_{rxy}$. Its total edge multiplicity is $D_r$ and every edge multiplicity lies in $\{0,1,2\}$.
\end{definition}

Suppose the residual blocks have active sizes $s$ and that point $x$ occurs in $u_{x,s}$ blocks of size $s$. Pair capacity through $\{r,x\}$ gives
\begin{equation}\label{eq:leave-degree}
 e_x=22-\sum_s(s-1)u_{x,s},
\end{equation}
where $e_x$ is the degree of $x$ in $L_r$. In addition,
\begin{equation}\label{eq:type-totals}
 \sum_xu_{x,s}=s n_s,
 \qquad
 \sum_xe_x=2D_r.
\end{equation}
Equations \eqref{eq:leave-degree}--\eqref{eq:type-totals} give a finite list of point-type multisets.

Let $M$ be the point-by-residual-block incidence matrix. Its Gram matrix $Q=MM^T$ is determined by the point types and the leave:
\begin{equation}\label{eq:gram}
 Q_{xx}=\sum_su_{x,s},
 \qquad
 Q_{xy}=2-\delta_{rxy}\quad(x\ne y).
\end{equation}
Consequently,
\begin{equation}\label{eq:rankcondition}
  \rank_{\mathbb F_p}(Q)\le q
\end{equation}
for every prime $p$, where $q$ is the number of residual blocks. When $q=12$, $M$ is square and
\begin{equation}\label{eq:detsquare}
  \det(Q)=\det(M)^2
\end{equation}
must be a nonnegative integer square. The verifier uses \eqref{eq:rankcondition} for $p=101,103,107$ and \eqref{eq:detsquare} only as necessary conditions. Cases passing these screens are retained.

For every surviving typed leave, the possible residual blocks form a finite linear factorization system. It requires the prescribed incidence count of every point in every size class, the prescribed pair multiplicities $2-\delta_{rxy}$, and the prescribed number of blocks of each size. An integral residual design would give a nonnegative integral solution. Each screened case is instead excluded by an integer-scaled rational Farkas vector whose inequalities are checked over every eligible residual block.

\begin{proposition}[Certified local screen]\label{prop:localscreen}
For each local parameter tuple invoked below, the program \texttt{local\_screen\_general.cpp} enumerates every point-type distribution satisfying \eqref{eq:leave-degree}--\eqref{eq:type-totals}, every loopless leave multigraph with edge multiplicities at most two and the prescribed degree sequence, and every case surviving \eqref{eq:rankcondition} or \eqref{eq:detsquare}. The Python verifier then checks an exact Farkas certificate for every retained factorization case.
\end{proposition}

The proposition is a finite computational statement. Its implementation is deliberately small: the C++ source is approximately four kilobytes, and the certificate checker uses only the Python standard library. The complete enumeration counts used below are summarized in \cref{tab:localcounts}.

\begin{table}[ht]
\centering
\small
\begin{tabular}{@{}llrrr@{}}
\toprule
Branch & Local parameters & type multisets & leaves & certified survivors\\
\midrule
Common $D=7$ & $(a,b)=(5,5)$ & 19 & 8,641 & 0\\
Common $D=7$ & $(a,b)=(8,3)$ & 37 & 6,733 & 4\\
$5^1 6^{14}7^7$, inside & $(a,b)=(6,4)$ & 508 & 43,715 & 0\\
$5^1 6^{14}7^7$, inside & $(a,b)=(3,6)$ & 6 & 534 & 0\\
$5^2 6^{12}7^8$, $c=2$ & $(10,1),(7,3),(4,5)$ & -- & -- & 195+217+28\\
$5^3 6^{10}7^9$, $c=3$ & $(8,2),(5,4),(2,6)$ & -- & -- & 4,746+656+3\\
$4^1 6^{13}7^8$, inside & $(8,3),(5,5)$ & 41+48 & 367+275 & 12+3\\
\bottomrule
\end{tabular}
\caption{Selected exact local-enumeration counts. A dash indicates that only the final screened-case count is used in the proof report.}
\label{tab:localcounts}
\end{table}

\section{The common deficit-seven branch}
Four of the remaining profiles may force a row outside all degree-five blocks with $D_r=7$. If $a,b$ denote the numbers of degree-six and degree-seven columns through $r$, then \eqref{eq:deficitformula} leaves
\[
 (a,b)=(5,5),(8,3),(11,1).
\]
The $(11,1)$ branch has no point-type distribution. For $(5,5)$, all $8{,}641$ leave graphs fail the Gram-rank condition. The $(8,3)$ branch has $37$ point-type distributions and $6{,}733$ leave graphs; four cases survive the Gram screen. Three have immediate exact local Farkas certificates.

The fourth case is fractionally feasible and therefore requires an integral refinement. The verifier enumerates all systems of its three residual size-six blocks, obtaining exactly $285$ systems. The full automorphism group preserving the typed leave is $S_2\times S_3\times S_6$, of order $8{,}640$, and partitions the systems into three orbits of sizes $15$, $180$, and $90$. The first two representatives have no completion by eight residual size-five blocks; the third has exactly four completions. For each of these four local completions, and for each of the four relevant global profiles, a separate rational completion certificate excludes the remaining columns. Thus sixteen exact completion certificates close the branch.

\begin{lemma}\label{lem:commonD7}
No marked row outside all exceptional small blocks can realize the common deficit-seven configuration used in the profiles $6^{16}7^6$, $5^1 6^{14}7^7$, $5^2 6^{12}7^8$, or $5^3 6^{10}7^9$.
\end{lemma}

\section{Elimination of the five final profiles}
We now combine residue baselines with \cref{prop:localscreen,lem:commonD7}.

\subsection{The profile $6^{16}7^6$}
Here $s=42$, so $\sum_rD_r=126$, and every $D_r\equiv2\pmod5$. If no row had deficit $2$ or $7$, then every deficit would be at least $12$, giving a total at least $156$. Every deficit-two local type is rejected by point-type or Gram screening, while every deficit-seven type is rejected by \cref{lem:commonD7}. Hence the profile is impossible.

\subsection{The profile $5^1 6^{14}7^7$}
Let $c\in\{0,1\}$ record membership in the unique degree-five block. Then $D_r\equiv2-c\pmod5$. The minimum cases $D=2$ for $c=0$ and $D=1$ for $c=1$ are locally impossible. The next baseline has deficit seven on the eight outside rows and six on the five inside rows, totaling
\[
 8\cdot7+5\cdot6=86.
\]
The required total is $111$, only five increments of size five larger, so some row attains this baseline. Outside rows are impossible by \cref{lem:commonD7}; the inside possibilities
\[
 (a,b)=(3,6),(6,4),(9,2),(12,0)
\]
are all rejected by exact local screening. Thus the profile is impossible.

\subsection{The profile $5^2 6^{12}7^8$}
Let $c\in\{0,1,2\}$ count the two degree-five blocks through the marked row. The initial cases $(c,D)=(0,2),(1,1),(2,0)$ are impossible. The next residue baseline has total
\[
 7\cdot13-10=81,
\]
whereas the required total is $96$. Thus a baseline row exists. The $c=0$ and $c=1$ branches reduce to earlier screens. For $c=2,D=5$, the possible values are
\[
 (a,b)=(10,1),(7,3),(4,5),(1,7).
\]
The last has no point types. The first three yield $195$, $217$, and $28$ screened cases, respectively, all excluded by exact Farkas certificates.

\subsection{The profile $5^3 6^{10}7^9$}
The initial $c=0,1,2$ cases are again impossible. The next baseline has total
\[
 7\cdot13-15=76,
\]
while the required total is $81$, so a baseline row exists. The branches $c=0,1,2$ have already been excluded. For $c=3,D=4$, the possibilities
\[
 (a,b)=(8,2),(5,4),(2,6)
\]
produce $4{,}746$, $656$, and $3$ screened local cases. Every one of the $5{,}405$ factorization systems has an exact integer-scaled Farkas certificate.

\subsection{The profile $4^1 6^{13}7^8$}
There are nine rows outside and four inside the degree-four block. Their minimum-residue baseline is
\[
 9\cdot2+4\cdot4=34,
\]
whereas the required total is $84$. If neither minimum occurred, the total would be at least $99$, so a row of deficit two outside or deficit four inside must occur. Outside deficit-two rows are impossible. For an inside deficit-four row, the residual exceptional block has size three. The endpoint branches have no point types; the two nontrivial branches leave $12$ and $3$ screened factorization cases, all excluded by exact certificates.

\begin{proposition}\label{prop:five}
None of the five profiles
\[
6^{16}7^6,
\quad5^1 6^{14}7^7,
\quad5^2 6^{12}7^8,
\quad4^1 6^{13}7^8,
\quad5^3 6^{10}7^9
\]
can occur in a $K_{3,3}$-free $13\times22$ matrix with $138$ ones.
\end{proposition}

\section{Proof of the main theorem}
\begin{proof}[Proof of \cref{thm:13}]
The witness in \cref{sec:witness-appendix} proves \eqref{eq:lower}. Suppose a $138$-one matrix existed. The penalty enumeration \eqref{eq:penaltybudget} places its degree multiset among $83$ profiles. Exact global Farkas certificates exclude $77$ profiles. \Cref{prop:nine} excludes $6^{18}7^3 9^1$, and \cref{prop:five} excludes the five remaining profiles. Therefore no $138$-one matrix exists. Deleting ones shows that no larger $K_{3,3}$-free matrix exists either, so the checked $137$-one construction is optimal.
\end{proof}

\section{Verification, reproducibility, and trust boundary}
The complete proof artifact is designed to separate discovery from verification.

\subsection{One-command verification}
From the repository root, run
\begin{verbatim}
cd proof
python3 verify_all.py
\end{verbatim}
The verifier compiles \texttt{local\_screen\_general.cpp} with a C++17 compiler if needed, then performs the following checks:
\begin{enumerate}[label=\arabic*.]
\item validates the $137$-one witness over all $286$ row triples;
\item independently re-enumerates the $83$ possible $138$-edge degree profiles;
\item replays $77$ global rational Farkas certificates;
\item checks the elementary exclusion of $6^{18}7^3 9^1$;
\item exhaustively regenerates every marked-row point-type and leave case used for the five final profiles;
\item verifies every local and completion certificate coefficient exactly.
\item runs the nested 12-by-18 through 12-by-22 witness scans and deletion-chain checks in \texttt{proof/z12\_18\_21/}.
\end{enumerate}
The expected final line is
\begin{verbatim}
ALL EXACT-VALUE VERIFICATION GATES PASSED
12-BY-18-THROUGH-22 GATE PASSED
\end{verbatim}
The frontier gate is run separately from its directory:
\begin{verbatim}
cd proof/frontier_closure
python3 verify_all.py
python3 mutation_tests.py
python3 frontier_propagation.py
\end{verbatim}
It verifies 13 global and 12 local frontier certificates, six explicit
witnesses, the five exact equalities above, and the $132\le Z(16,17)\le133$
interval. The optional dependency flag replays the bundled 12-row package.

\subsection{Certificate counts}
The complete development archive also retains an earlier certified $139$-edge exclusion. Across all retained gates, the checker verifies $125{,}983$ candidate-block coefficients for the historical $139$ proof, $390{,}039$ for the global $138$ reduction, and $5{,}262{,}392$ for the final local proof, totaling
\[
 5{,}778{,}414
\]
exact nonnegativity checks. The $139$ component is not logically needed for \cref{thm:13}; it is retained as an independent regression test.

\subsection{What is and is not trusted}
The proof depends on:
\begin{itemize}
\item exhaustive finite integer enumeration in the supplied C++ source;
\item Python integer and \texttt{Fraction} arithmetic;
\item explicit Farkas vectors stored in JSON;
\item the explicit lower-bound witness.
\end{itemize}
Floating-point linear programming was used during discovery to locate dual vectors. Before inclusion, these vectors were rationalized or integer-scaled, and the final checker verifies their defining inequalities from scratch. No floating-point solver status, MILP infeasibility flag, tolerance, or randomized search result is a premise.

The only symmetry reduction in the integral common-$D=7$ branch constructs all $8{,}640$ group elements explicitly, verifies that each preserves point types and leave multiplicities, and checks orbit coverage of all $285$ labeled systems. Repeated columns remain permitted throughout.

\subsection{Epistemic status}
The package has passed its internal exact-arithmetic gate and an adversarial audit targeting omitted degree profiles, hidden simplicity assumptions, incomplete leave enumeration, invalid rank or determinant implications, symmetry loss, and numerical tolerance. Independent reproduction, a second implementation of the local enumerator, and expert peer review remain desirable. Accordingly, this manuscript reports a complete reproducible computer-assisted proof, not a peer-reviewed consensus result.

\section{Discussion}
The proof illustrates a useful division of labor for finite extremal problems. A coarse global relaxation disposes of most degree patterns efficiently. The few profiles surviving the relaxation are not random hard instances: they are nearly balanced around the zero-penalty degrees six and seven, and their small unused triple capacity forces strong congruences on marked-row deficits. Passing to the leave of a marked row converts those congruences into a compact residual packing problem whose Gram matrix is almost determined.

The method is potentially reusable. For nearby $K_{3,3}$-free matrix problems, one may seek a supporting line for $d\mapsto\binom d3$, enumerate low-penalty degree profiles, apply global dual separation, and reserve exact leave enumeration for the integrality-only residue. The computational burden is then concentrated in small, transparent certificate families rather than a monolithic SAT or MILP run.

\sloppy
A literature and repository audit dated 3 August 2026 located no prior
publication or public preprint reporting the exact values
$\Z(12,18,3,3)=108$, $\Z(12,19,3,3)=114$, $\Z(12,20,3,3)=120$, or
$\Z(12,21,3,3)=126$. The audit treats the $n=22$ value and the
3-$(12,6,2)$ design as prior/classical material, and it records the scope and
limitation of the search in the accompanying 12-row package.

For $\Z(13,22,3,3)$, the 137-edge construction was previously reported while
the exact upper bound is supplied here. These priority statements are
necessarily provisional and should be rechecked during peer review.

The frontier audit dated 3 August 2026 located no prior indexed report of
$\Z(13,18)=116$, $\Z(14,17)=118$, $\Z(14,18)=124$, $\Z(15,17)=126$, or
$\Z(15,18)=132$, nor of the 132-edge $16\times17$ construction. This is a
dated search finding rather than a universal nonpublication guarantee.
\fussy

\section*{Acknowledgements}
OpenAI GPT 5.6 Sol High was used for assistance with exploratory reasoning,
implementation, verification, manuscript integration, and packaging. The
author is solely responsible for the mathematical claims and final manuscript.
No heuristic solver status is used as a theorem.

\appendix
\section{The nested $12$-row witness family}
\label{app:12-design}
The following 22 supports form the verified classical 3-$(12,6,2)$ design.
Every triple of row labels occurs in exactly two supports. The first $n$
supports therefore give the lower-bound witness for each $18\le n\le22$.
The machine-readable CSV files under \texttt{proof/z12\_18\_21/data/} are the
load-bearing representation used by the verifier.
\[
\begin{array}{r l}
1&\{2,7,8,9,11,12\}\\
2&\{1,3,4,5,6,10\}\\
3&\{1,2,3,4,8,9\}\\
4&\{2,4,5,8,10,11\}\\
5&\{1,4,6,7,8,11\}\\
6&\{3,5,6,8,9,11\}\\
7&\{1,2,3,5,7,11\}\\
8&\{1,2,6,9,10,11\}\\
9&\{1,2,5,6,8,12\}\\
10&\{3,4,5,7,8,12\}\\
11&\{1,3,6,7,9,12\}\\
12&\{1,2,4,7,10,12\}\\
13&\{2,3,5,9,10,12\}\\
14&\{4,6,8,9,10,12\}\\
15&\{2,3,4,6,11,12\}\\
16&\{1,4,5,9,11,12\}\\
17&\{5,6,7,10,11,12\}\\
18&\{1,3,8,10,11,12\}\\
19&\{1,5,7,8,9,10\}\\
20&\{2,3,6,7,8,10\}\\
21&\{2,4,5,6,7,9\}\\
22&\{3,4,7,9,10,11\}
\end{array}
\]

\section{The $137$-edge witness}\label{app:witness}
\label{sec:witness-appendix}
Rows are labeled $1,\dots,13$. The following $22$ column supports have total size $137$:
\[
\begin{array}{@{}ll@{}}
B_{1}=\{2,4,6,7,11,13\} & B_{2}=\{3,6,7,8,10,13\}\\
B_{3}=\{3,4,6,8,11,12\} & B_{4}=\{1,3,8,9,10,12\}\\
B_{5}=\{1,3,4,6,9,13\} & B_{6}=\{1,2,3,8,11,13\}\\
B_{7}=\{1,2,4,6,8,10\} & B_{8}=\{1,2,3,6,7,12\}\\
B_{9}=\{2,7,8,10,11,12\} & B_{10}=\{2,6,8,9,12,13\}\\
B_{11}=\{2,3,4,7,8,9\} & B_{12}=\{2,3,4,10,12,13\}\\
B_{13}=\{1,2,7,9,10,13\} & B_{14}=\{2,3,6,9,10,11\}\\
B_{15}=\{3,7,9,11,12,13\} & B_{16}=\{1,3,4,7,10,11\}\\
B_{17}=\{1,2,4,9,11,12\} & B_{18}=\{1,5,6,10,11,12,13\}\\
B_{19}=\{1,5,6,7,8,9,11\} & B_{20}=\{4,5,8,9,10,11,13\}\\
B_{21}=\{1,4,5,7,8,12,13\} & B_{22}=\{4,5,6,7,9,10,12\}\\
\end{array}
\]
The first seventeen blocks have size six and the final five have size seven. The accompanying verifier checks that each row triple occurs in at most two blocks.

\sloppy\raggedright

\fussy
\end{document}